\documentclass[12pt]{amsart}
\usepackage{amsmath,amscd,latexsym,verbatim,amssymb}
\usepackage{times}       
 \newcommand{\resp}{{\it resp.} }
\newcommand{\cf}{{\it cf.} }
\newcommand{\ie}{{\it i.e.} }
\newcommand{\eg}{{\it e.g.} }

\newcommand{\loccit}{{\it loc. cit.} }

\newcommand{\Q}{\mathbb{Q}}
\newcommand{\R}{\mathbb{R}}
\newcommand{\C}{\mathbb{C}}
  \newcommand{\Z}{\mathbb{Z}}

  \newcommand{\sA}{{\mathcal{A}}}

\newcommand{\sD}{{\mathcal{D}}}

\newcommand{\sH}{{\mathcal{H}}}

\newcommand{\sL}{{\mathcal{L}}}

\newcommand{\sF}{{\mathcal{F}}}
\newcommand{\sM}{{\mathcal{M}}}

\newcommand{\sO}{{\mathcal{O}}}

\newcommand{\sV}{{\mathcal{V}}}
\newcommand{\sW}{{\mathcal{W}}}
\newcommand{\sX}{{\mathcal{X}}}

\newcommand{\inj}{\hookrightarrow}

\newcommand{\surj}{\rightarrow\!\!\!\!\!\rightarrow}

\newcommand{\Spec}{\operatorname{Spec}}

\newcounter{spec}

\swapnumbers

\newtheorem{thm}{Theorem}[subsection]
\newtheorem{lemma}[thm]{Lemma}

\newtheorem{sco}[thm]{Scholium}
\newtheorem{por}[thm]{Porism}
\theoremstyle{definition}
\newtheorem{defn}[thm]{Definition}

\newtheorem{rem}[thm]{Remark}
\newtheorem{rems}[thm]{Remarks}

\numberwithin{equation}{section}

 at10pt
 at12pt

\renewcommand{\qed}{\hfill $\square$\medskip}
  
\begin{document}

\title[On the Kodaira-Spencer map of abelian schemes.]{On the Kodaira-Spencer map of abelian schemes }
 
\author{Y. Andr\'e}
 
  \address{Institut de Math\'ematiques de Jussieu, 4 place Jussieu, 75005 Paris France.}
\email{yves.andre@imj-prg.fr}
\keywords{abelian scheme, Kodaira-Spencer map, monodromy, Shimura variety, automorphic vector bundle}
 \subjclass{11G, 14K, 14M}
  \begin{abstract}   Let $A$ be an abelian scheme over a smooth affine complex variety $S$, $\varOmega_A$ the $\sO_S$-module of $1$-forms of the first kind on $A$, $\sD_S\varOmega_A$ the $\sD_S$-module spanned by $\varOmega_A$ in the first algebraic De Rham cohomology module, and $\theta_\partial: \varOmega_A \to  \sD_S\varOmega_A/\varOmega_A$ the Kodaira-Spencer map attached to a tangent vector field $\partial$ on $S$. We compare the rank of $\sD_S\varOmega_A/\varOmega_A$ to the maximal rank of $\theta_\partial$ when $\partial$ varies: we show that both ranks do not change when one passes to the ``modular case", \ie when one replaces $S$ by the smallest weakly special subvariety of $\sA_g$ containing the image of $S$ (assuming, as one may up to isogeny, that $A/S$ is principally polarized); we then analyse the ``modular case" and deduce, for instance, that {\it for any abelian pencil of relative dimension $g$ with Zariski-dense monodromy in $Sp_{2g}$}, {\it the derivative with respect to a parameter of a non zero abelian integral of the first kind is never of the first kind}. 
  \end{abstract}
\maketitle

 \begin{sloppypar}

     \medskip 
  This paper deals with abelian integrals depending algebraically on parameters and their derivatives with respect to the parameters. Since the nineteenth century, it has been known that differentiation with respect to parameters does not preserve abelian integrals of the first kind in general. 
 
 We study this phenomenon in the language of modern algebraic geometry, i.e. in terms of the algebraic De Rham cohomology $\sO_S$-module $\sH^1_{dR}(A/S)$ attached to an abelian scheme $A$ of relative dimension $g$ over a smooth $\C$-scheme $S$, its submodule $\varOmega_A$ of forms of the first kind on $A$, the Gauss-Manin connection $\nabla$ and the associated Kodaira-Spencer map $\theta $, \ie 
 the $\sO_S$-linear map $\,T_S \otimes \varOmega_A \stackrel{\theta}{\to} \sH^1_{dR}(A/S)/\varOmega_A\,$ induced by $\nabla$.
 
 We introduce and compare the following (generic) ``ranks":
 
  \smallskip $r = r(A/S) = {\rm rk}  \, \sD_S \varOmega_A/\varOmega_A,\; $
 
  \smallskip  $ r' = r'(A/S)  = {\rm rk}\,  \theta ,\; $
 
  \smallskip  $\displaystyle r'' = r''(A/S)  = \max_\partial  {\rm rk}\, \theta_\partial, \;$ 
 
    where $\partial$ runs over local tangent vector fields on $S$ (of course, $r''= r'$ when $S$ is a curve).
 
 \smallskip One has $r''\leq r'\leq r\leq g$ and these inequalities may be strict, even if there is no isotrivial factor (\ref{exe}). On the other hand, these ranks are insensitive to dominant base change, and depend only on the isogeny class of $A/S$  (\ref{l2}).  In particular, one may assume that $A/S$ is principally polarized and (replacing $S$ by an etale covering) admits a level $n\geq 3$ structure.

\medskip We prove that {\it $r$ and $r''$ are unchanged if one passes to the ``modular case"}, \ie if one replaces $S$ by the smallest weakly special (= totally geodesic) subvariety of the moduli space $\sA_{g,n}$ containing the image of $S$, and $A$ by the universal abelian scheme on $S$ (\ref{t5}). 

We prove that {\it  $r = r'$ in the ``modular case"}, \ie when $S$ is a weakly special subvariety of $\sA_{g,n}$ (\ref{t6}).  

\medskip  We then study the ``PEM case", \ie the case when the connected algebraic monodromy group is maximal with respect to the polarization and the endomorphisms, and emphasize the ``restricted PEM case", \ie where we moreover assume that if the center $F $ of ${\rm End}\, A \otimes \Q$ is a CM field, then $\varOmega_A$ is a free $F\otimes_\Q \sO_S$-module (\ref{pe}, \ref{rpe}); this includes, of course, the case when the algebraic monodromy group is $Sp_{2g}$.

Building on the previous results, we show that one has {\it $r'' = r' = r= g$ in the restricted PEM case} (\ref{t7}). 
 If moreover $S$ is a curve, we show that {\it the derivative (with respect to a parameter) of a non zero abelian integral of the first kind is never of the first kind} (\ref{t8}). 

  \medskip  Our methods are inspired by B. Moonen's paper \cite{Mo}; we exploit the ``bi-algebraic" properties of the Kodaira-Spencer map in the guise of a theorem of ``logarithmic Ax-Schanuel type" for tangent vector bundles (\ref{t4}). 
  
 \medskip Since the problems under study occur in various parts of algebraic geometry and diophantine geometry, we have tried to make the results more accessible by including extended reminders: section 1 about algebraic De Rham cohomology of abelian schemes, Gauss-Manin connections and Kodaira-Spencer maps; subsections 3.1 to 3.4 about weakly special subvarieties of connected Shimura varieties, relative period torsors, and automorphic bundles.

 \section{Preliminaries}

 \subsection{Invariant differential forms}  Let $S$ be a smooth connected scheme over a field $k$ of characteristic zero. 
 
Let $f : G {\to} S$ be a smooth commutative group scheme; we denote by $m : G\times_S G \to G$ the group law and by $e : S \to G$ the unit section. The {\it invariant differential $1$-forms} on $G$ are those satisfying $m^\ast \omega = p_1^\ast \omega + p_2^\ast \omega$ (where $p_1, p_2$ denote the projections); they form a locally free $\sO_S$-module denoted by $\varOmega_G$, naturally isomorphic to $e^\ast \Omega^1_{G/S}$  and to $f_{ \ast} \Omega^1_{G/S}$, and $\sO_S$-dual to the Lie algebra ${\rm Lie}\, G$. One has $f^\ast  \varOmega_G \cong  \Omega^1_{G/S}$. Moreover, invariant differential $1$-forms are closed \cite[3.5]{MG}\cite[1.2.1]{C}.

Let us consider the special case when $G=A$ is an abelian scheme of relative dimension $g$, or $G=A^\natural$ universal vectorial extension of $A$ (Rosenlicht-Barsotti, \cf \eg \cite[I]{MM}). Recall that $Ext(A, \mathbb G_a) \cong  R^1f_\ast \sO_A$ (using the fact that any rigidified $\mathbb G_a$-torsor over an $S$-abelian scheme has a canonical $S$-group structure), so that $A^\natural$ is an extension of $A$ by the vector group attached to the dual of  $R^1f_\ast \sO_A$, which is a locally free $\sO_S$-module of rank $g$. The projection $A^\natural\to A$ gives rise to an exact sequence of locally free $\sO_S$-modules
\begin{equation}\label{eq0} 0 \to \varOmega_A \to \varOmega_{A^\natural} \to  R^1f_\ast \sO_A \to 0,\end{equation}
  in a way compatible with base change $S'\to S$.
 On the other hand, if $A^t := Pic^0(A)$ denotes the dual abelian scheme,  $\varOmega_{A^{t}}$ is naturally dual to $R^1f_\ast \sO_A$ (Cartier), and $\varOmega_{A^{t \natural}}$ is naturally dual to $\varOmega_{A^\natural}$ in such a way that the exact sequence \eqref{eq0} is dual to corresponding exact sequence for $A^t$ \cite[1.1.1]{C}.

  \subsection{Algebraic De Rham cohomology}  The {\it first algebraic De Rham cohomology} $\sO_S$-module $  \sH^1_{dR}(G/S)$ is the hypercohomology sheaf ${\bf R}^1f_\ast(\Omega^\ast_{G/S}, d)$. Assuming $S$ affine, it can be computed \`a la \v Cech using an affine open cover $\mathcal U$ of $G$ and taking as coboundary map on $C^p(\mathcal U, \Omega^q_{G/S})$ the sum of the \v Cech coboundary and $(-)^{p+1}$ times the exterior derivative $d$. In particular, since invariant differential forms are closed, there is a canonical $\sO_S$-linear map 
   $\varOmega_G \to  \sH^1_{dR}(G/S)$. 
  
 In case $G= A $ is an abelian scheme, and $A^\natural$ its universal vectorial extension, it turns out that the canonical morphisms $$\varOmega_{A^{\natural}} \to  \sH^1_{dR}(A^{\natural}/S) \leftarrow  \sH^1_{dR}(A/S) $$ are isomorphisms \cite[1.2.2]{C}. The exact sequence 
 \eqref{eq0} thus gives rise to an exact sequence of locally free $\sO_S$-modules
 \begin{equation}\label{eq1} 0 \to \varOmega_A = f_\ast \Omega^1_{A/S} \to  \sH^1_{dR}(A/S) \to  R^1f_\ast \sO_A = \varOmega^\vee_{A^t} \to 0, \end{equation} in a way compatible with base change $S'\to S$ and with duality $A\mapsto A^t$ (\cf also \cite[8.0]{K1}; $  f_\ast \Omega^1_{A/S}$ and  $R^1f_\ast \sO_A $ are the graded pieces $gr^1$ and $gr^0$ of the Hodge filtration of $\sH^1_{dR}(A/S)$ respectively). 
 
 Any polarization of $A$ endows the rank $2g$ vector bundle $\sH^1_{dR}(A/S)$ with a symplectic form, for which $\varOmega_A$ is a lagrangian\footnote{\ie isotropic of rank $g$.} subbundle, and the exact sequence \eqref{eq1} becomes autodual. 
 
 \smallskip When $S = \Spec k$,  $\sH^1_{dR}(A/S)$ can also be interpreted as the space of differential of the second kind (\ie closed rational $1$-forms which are Zariski-locally sums of a regular $1$-form and an exact rational form) modulo exact rational $1$-forms. For any rational section $\tau$ of $A^\natural \to A$ and any $\eta\in  \varOmega_{A^{\natural}}$, $\tau^\ast\eta$ is of the second kind and depends on $\tau$ only up to the addition of an exact rational $1$-form. 
 
 \medskip  In the sequel, we abbreviate $\sH^1_{dR}(A/S)$ by $\sH$. 
  
 \subsection{Gauss-Manin connection}\label{GM}  Since the nineteenth century, it has been known that differentiating abelian integrals with respect to parameters leads to linear differential equations, the prototype being the Gauss hypergeometric equation in the variable $t$ satisfied by $\int_1^\infty z^{a-c}(1-z)^{c-b-1} (1-tz)^{-a}dz$. Manin gave an algebraic construction of this differential module (in terms of differentials of the second kind), later generalized by Katz-Oda and others to the construction of the Gauss-Manin connection on algebraic De Rham cohomology of any smooth morphism $X\to S$. 
 
 \smallskip Let as before $A\stackrel{f}{\to} S$ be an abelian scheme of relative dimension $g$ over a smooth connected $k$-scheme $S$. If $k=\mathbb C$, the {\it Gauss-Manin connection} is determined by its analytification $\nabla^{an}$, whose dual is the unique analytic connection on $(\sH^\vee)^{an}$ which kills the period lattice 
   \begin{equation}\label{eq6}\ker \exp_{A} \cong \ker \exp_{A^\natural} \subset ({\rm Lie}\, A^{\natural})^{an} =  (\varOmega_{A^{\natural}}^\vee)^{an} = (\sH^\vee)^{an} .\end{equation}
   
   The formation of $(\sH, \nabla)$ is compatible with base change $S'\to S$ and with duality $A\mapsto A^t$. It is contravariant in $A$, and $S$-isogenies lead to isomorphisms between Gauss-Manin connections.

\smallskip  If $S$ is affine and $\varOmega_A$ and $\varOmega_{A^t}$ are free, let us take a basis $\omega_1, \ldots, \omega_g$ of  $\varOmega_A$ and complete it into a basis $\omega_1, \ldots, \omega_g, \eta_1, \ldots, \eta_g$ of  $\sH$. Pairing with a basis $\gamma_1, \ldots, \gamma_{2g}$ of the period lattice on a universal covering $\tilde S$ of $S^{an}$, one gets a full solution matrix  
   \begin{equation}\label{eq7} Y = \begin{pmatrix}\Omega_1 & {\rm N}_1\\ \Omega_2 & {\rm N}_2 \end{pmatrix}\in  M_{2g}(\sO(\tilde S)) \end{equation} for $\nabla$ (with $(\Omega_1)_{ij} = \int_{\gamma_i} \omega_j$, etc...). This reflects into a family of differential equations\footnote{we write the matrix of $\nabla_\partial$ on the right in order to let the monodromy act on the left on $Y$. This convention has many advantages.
    In particular, it is independent of the choice of $\gamma_1, \ldots, \gamma_{2g}$.} 
     \begin{equation}\label{eq8} \partial Y =   Y \begin{pmatrix}R_\partial & S_\partial\\ T_\partial & U_\partial \end{pmatrix}, \end{equation}  where $R_\partial , S_\partial, T_\partial , U_\partial  \in  M_{g}(\sO(S))$\footnote{the fact that these matrices have entries in $\sO(S)$ rather than $\sO(S^{an})$ reflects the algebraic nature of the Gauss-Manin connection. Alternatively, it can be deduced from the next sentence.}  depend $\sO(S)$-linearly on the derivation $\partial \in \Gamma T_S.$  
      
     It is well-known that the Gauss-Manin connection is regular at infinity (\cf \eg \cite[14.1]{K1}), hence its $\sD$-module theoretic properties are faithfully reflected by monodromy theoretic properties. 
     
\rem  The Katz-Oda algebraic construction of $\nabla$, in the case of $\sH^1_{dR}(A/S)$, goes as follows \cite[1.4]{K2}. From the exact sequence
 \begin{equation}\label{eq2} 0 \to f^\ast \Omega^1_{S/k}  \to \Omega^1_{A/k}  \to   \Omega^1_{A/S} \to 0,\end{equation}
 passing to exterior powers, one gets the exact sequence of $k$-linear complexes of $\sO_A$-modules 
  \begin{equation}\label{eq3} 0 \to f^\ast \Omega^1_{S/k}\otimes \Omega^{\ast-1}_{A/S}  \to \Omega^\ast_{A/k}/( f^\ast \Omega^2_{S/k}\otimes \Omega^{\ast-2}_{A/S})  \to   \Omega^\ast_{A/S} \to 0.\end{equation}
  Then $\nabla$ is a coboundary map in the long exact sequence for ${\bf R}^\ast f_\ast$ applied to \eqref{eq3}, that is
   \begin{equation}\label{eq4}   {\bf R}^1 f_\ast \Omega^\ast_{A/S} \stackrel{\nabla}{\to}  {\bf R}^2 f_\ast (f^\ast \Omega^1_{S/k}\otimes \Omega^{\ast-1}_{A/S}) =  \Omega^1_{S/k}\otimes {\bf R}^1 f_\ast \Omega^\ast_{A/S},\end{equation} and can be computed explicitly \`a la \v Cech, \cf \cite[3.4]{K1}.
  One checks that this map satisfies the Leibniz rule and the associated map
   \begin{equation}\label{eq5} T_{S} = (\Omega^1_{S/k})^\vee \stackrel{\partial \mapsto \nabla_\partial}{\to} End_k  \sH \end{equation} 
 respects Lie brackets, so that $\nabla$ corresponds to a $\sD_S$-module structure on $\sH$   (here $\sD_S$ denotes the sheaf of rings of differential operators on $S$, which is generated by the tangent bundle $T_{S }$).  In fact, it can also be interpreted as the first higher direct image of $ \sO_A $ in the $\sD$-module setting (\cf \eg\cite[4]{CaF} for an algebraic proof). 
  
 \smallskip An alternative and more precise construction of $\nabla$, which avoids homological algebra, consists in endowing $A^{\natural}$ with the structure of a commutative algebraic $\sD$-group, which automatically provides a connection on (the dual of) its Lie algebra \cite[3.4, H5]{BP}\cite[6]{B}.

     \subsection{Kodaira-Spencer map}  
     The Gauss-Manin connection does not preserves the subbundle $\varOmega_A\subset \sH$ in general. The composed map
      \begin{equation}\label{eq9} \varOmega_A\inj \sH \stackrel{\nabla}{\to} \Omega^1_S\otimes \sH \surj   \Omega^1_S\otimes  (\sH/  \varOmega_{A }) =  \Omega^1_S\otimes    \varOmega^\vee_{A^t} \end{equation} is the {\it Kodaira-Spencer map} (or Higgs field). Like the Gauss-Manin connection, its formation commutes with base-change. Unlike the Gauss-Manin connection, it is an $\sO_S$-linear map (also called the Higgs field of $A/S$ \cite{VZ}).    
      
      \rem  This map can be interpreted as a coboundary map in the long exact sequence for ${ R}^\ast f_\ast$ applied to \eqref{eq2}, and computed explicitly \`a la \v Cech, \cf \cite[3.4]{K1}\cite[1.3]{K2}.

 \smallskip  It can be rewritten as the map
         \begin{equation}\label{eq10}  \theta: \; T_S \otimes_{\sO_S}   \varOmega_{A } \;  {\to} \;\varOmega^\vee_{A^t } = {\rm Lie}\, A^t. \end{equation}
   If $\sD_S^{\leq 1}\subset \sD_S$ denotes the subsheaf of differential operators of order $\leq 1$ on $S$,  and $\sD_S\varOmega_A\subset \sH$ the sub-$\sD_S$-module generated by $\varOmega_A$ in $\sH = \sH^1_{dR}(A/S)$. one has 
 \begin{equation}\label{eq10'} {\rm Im}\, \theta = \sD_S^{\leq 1}  \varOmega_{A }/  \varOmega_{A } \subset   \sD_{S }  \varOmega_{A }/  \varOmega_{A } \subset \sH/    \varOmega_{A } = {\rm Lie}\, A^t. \end{equation}
   
  \smallskip The Kodaira-Spencer map can also be rewritten as the map
     \begin{equation}\label{eq11}  T_S \;\stackrel{\partial \mapsto \theta_\partial}{\to} \; {\rm Lie}\, A \otimes {\rm Lie}\, A^t, \end{equation}
    which is invariant by duality $A\mapsto A^t$ \cite[9.1]{CF}; if $A$ is polarized, it thus gives rise to a map  
       \begin{equation}\label{eq12}  T_S \; \stackrel{\partial \mapsto \theta_\partial}{\to}  \; S^2 {\rm Lie}\, A \cong {\rm Hom_{sym}}(  \varOmega_{A } ,   \varOmega_{A } ^\vee) . \end{equation}
          In the situation and notation of the end of \S  \ref{GM}, {\it the matrix of $\theta_\partial$ is $T_\partial$ (which is a symmetric matrix if one chooses the basis  $\omega_1,\ldots , \eta_g$ to be symplectic)}.

 \begin{rems}  $i)$  Here is another interpretation of $\theta_\partial$ in terms of the universal vectorial extension $A^\natural$, assuming $S$ affine  \cite[9]{CF}: for any $\omega\in \Gamma\varOmega_A$, pull-back the exact sequence of vector bundles associated to \eqref{eq2} by the morphism $\sO_A \to \Omega^1_{A/S}$ corresponding to $\omega$ and get an extension of $A$ by the vector group attached to $\Omega^1_S$, so that the morphism from $A^\natural$ to this vectorial extension gives rise, at the level of invariant differential forms, to a morphism $\varOmega_{A^t} \to \Omega^1_S$; thus to any $\omega$ and any $\partial \in \Gamma T_S = {\rm Hom}(\Omega^1_S, \sO_S)$, one gets an element of $ \varOmega^\vee_{A^t} $, which is nothing but $\theta_\partial\cdot \omega\,$.
  
  \smallskip $ii)$ The following equivalences are well-known: 
      
       $A/S$   is isotrivial $\Leftrightarrow  \theta = 0  \Leftrightarrow\sD_S\varOmega_A= \varOmega_A  \Leftrightarrow \nabla$ is isotrivial (\ie has finite monodromy). 
       
    Remembering that the Kodaira-Spencer map commutes to base-change, the only non trivial implications are:  
   $   \nabla$ isotrivial $\Rightarrow  \sD_S\varOmega_A= \varOmega_A$, and  $ \,\theta = 0 \Rightarrow A/S$ isotrivial. 
    The first implication comes from Deligne's ``th\'eor\`eme de la partie fixe" \cite[4.1.2]{D}. An elementary proof of the second one will be given below (\ref{bial}).   
    
    \smallskip $iii)$ In contrast to $\sD_S\varOmega_A$, $\sD^{\leq 1}_S\varOmega_A$ is not locally a direct factor of $\sH$ in general: at some points $s\in S$ the rank of $\theta_s$ may drop (see however th. \ref{t6}). In fact, the condition that the rank of $\theta_s$ is constant is very restrictive: for instance, if $S$ is a proper curve, the condition that $\theta$ is everywhere an isomorphism is equivalent to the condition that the Arakelov inequality ${\rm deg}\,\varOmega_{A } \leq \frac{g}{2} {\rm deg}\,\Omega^1_S$ is an equality, and implies that $A/S$ is a modular family, parametrized by a Shimura curve \cite{VZ}. 
    \end{rems}

    \subsubsection{}  Since the $\sO_S$-module $\sH/\sD_S\varOmega_A$ carries a $\sD_S$-module structure, it is locally free \cite[8.8]{K1}, hence $\sD_S\varOmega_A$ is locally a direct summand of $\sH$. In fact, by Deligne's semisimplicity theorem \cite[4.2.6]{D}, $\sD_S\varOmega_A$ is even a direct factor of $\sH$ (as a $\sD_S$-module, hence as a vector bundle).
   
   \begin{lemma}\label{l1} The formation of $ \sD_S\varOmega_A$ commutes with dominant base change $S'\stackrel{\pi}{\to} S$ (with $S'$ smooth connected).
   \end{lemma}  
   
   \begin{proof} Since $\sH$ commutes with base-change and $ \sD_S\varOmega_A$ is locally a direct summand, it suffices to prove the statement after restricting $S$ to a dense affine open subset. In particular, one may assume that $\pi$ is a flat submersion, so that $T_{S'}\to \pi^\ast T_S$ and $\sD_{S'} \to \pi^\ast \sD_S$ are epimorphisms, and $ \sD_{S'}\varOmega_{A_{S'}} = \pi^\ast\sD_{S}\pi^\ast\varOmega_{A }= \pi^\ast ( \sD_S\varOmega_A)$.
   \end{proof}

   \subsubsection{}     As in the introduction, let us define
\begin{equation}\label{1} r = r(A/S) : = {\rm rk}  \, \sD_S \varOmega_A/\varOmega_A,\; \end{equation}
\begin{equation}\label{2}  r' = r'(A/S) : = {\rm rk}\,  \theta =  {\rm rk}  \, \sD^{\leq 1}_S \varOmega_A/\varOmega_A,\; \end{equation}
 \begin{equation}\label{3} r'' = r''(A/S) : = \max_\partial  {\rm rk}\, \theta_\partial,\; \end{equation}
     where $\partial$ runs over local tangent vector fields on $S$ (and $\rm rk$ denotes a generic rank). 

       \begin{lemma}\label{l2} These are invariant by dominant base change $S'\stackrel{\pi}{\to} S$ (with $S'$ smooth connected), and depend only on the isogeny class of $A/S$. 
   \end{lemma}    
       
        \begin{proof} For $r$, this follows from the previous lemma. Its proof also shows that $ \sD^{\leq 1}_S\varOmega_A$ commutes with  base change by flat submersions, which settles the case of $r'$. For $r''$, we may assume that $S$ and $S'$ are affine, that $T_S$ is free and $T_{S'} = \pi^\ast T_S$, and pick a basis $\partial_1, \ldots, \partial_d$ of tangent vector fields; the point is that $ \max_{\lambda_i}\,  {\rm rk}\,\sum \lambda_i \theta_{\partial_i}$ is the same when the $\lambda_i$'s run in $\sO(S)$ or in $\sO(S')$ (consider the $\theta_{\partial_i}$'s as matrices and note that each minor determinant is a polynomial in the $\lambda_i$'s). 
        
        The second assertion is clear since any isogeny induces an isomorphism at the level of $(\sH, \nabla)$.  \end{proof}

         \begin{lemma}\label{l2} 
         \begin{enumerate} 
         \item $r''= g\,$ holds if and only if there exists a local vector field $\partial$ such that $\theta_\partial . \omega \neq 0$ for every non-zero $ \omega  \in \Gamma\varOmega_{A }$.
         \item $r'=g\,$  holds if and only if  for every non-zero $ \omega  \in \Gamma\varOmega_{A } $, there exists a local vector field $\partial$ such that$ \;\theta_\partial . \omega \neq 0$.
         \end{enumerate}        
    \end{lemma}
      
  \begin{proof} The first equivalence is immediate, while the second uses the symmetry of \eqref{eq11}: assuming $A$ polarized, and after restricting $S$ to a dense open affine subset, one has $r'=g\Leftrightarrow \forall \omega  \in \Gamma\varOmega_{A }\setminus 0, \; \exists \eta \in \Gamma\varOmega_{A }, \exists \partial \in \Gamma T_S, \;(\theta_\partial . \omega) \cdot \eta \neq 0$.  Since $(\theta_\partial . \eta) \cdot \omega = (\theta_\partial . \omega) \cdot \eta$, one gets  $\forall \omega  \in \Gamma\varOmega_{A }\setminus 0, \; \exists \eta \in \Gamma\varOmega_{A }, \exists \partial \in \Gamma T_S, \;(\theta_\partial . \eta) \cdot \omega \neq 0$. \end{proof}

      \section{Automorphic vector bundles and bi-algebraicity}  
      
           \subsection{Bi-algebraicity of the Kodaira-Spencer map}  
            \subsubsection{}\label{bial}  Let $ \sA_{g, n}$ be the moduli scheme of principally polarized abelian varieties of dimension $g$ with level $n$ structure ($n\geq 3$), and let $\sX \to  \sA_{g, n}$ be the universal abelian scheme. 
       
 The universal covering of $\sA_{g, n}^{an}$ is the Siegel upper half space $\mathfrak H_g$. We denote by $j_{g,n}: \mathfrak H_g \to \sA_{g, n}^{an}$ the uniformizing map (for $g=n=1$, this is the usual $j$-function). The pull-back of the dual of the period lattice $  \ker \exp_{\sX}$ on $\mathfrak H_g$ is a {\it constant symplectic lattice} $\Lambda$. On $\mathfrak H_g$, the Gauss-Manin connection of $ \sX/\sA_{g, n}$ becomes a trivial connection with solution space $\Lambda_\C$.

 On the other hand, $\mathfrak H_g$  is an (analytic) open subset of its ``compact dual" $\mathfrak H_g^\vee$, which is the grassmannian of lagrangian subspaces $V \subset \Lambda^\vee_\C$ (\ie isotropic subspaces of dimension $g$):  the lagrangian subspace $V_\tau$ corresponding to a point $\tau \in \mathfrak H_g $ is $\varOmega_{\sX_{j_{g,n}(\tau)}}\subset  \sH^1_{dR}(\sX_{j_{g,n}(\tau)})  \cong \Lambda_\C^\vee  $ (note that the latter isomorphism depends on $\tau$, not only on $j_{g,n}(\tau)$). The grassmannian $\mathfrak H_g^\vee$ is a homogeneous space for $Sp(\Lambda_\C)$ (in block form $\begin{pmatrix}A& B\\ C &D \end{pmatrix} $ sends $\tau\in \mathfrak H_g$ to $(A\tau + B)(C\tau + D)^{-1}$).
  The vector bundle $ j_{g,n}^\ast {\rm Lie} {\sX }$  is the restriction to $\mathfrak H_g$ of the tautological vector bundle $\sL$ on the lagrangian grassmannian $\mathfrak H_g^\vee$. 
 
 \subsubsection{}\label{ks}   In this universal situation, the Kodaira-Spencer map (in the form of \eqref{eq12}) is an isomorphism 
   \begin{equation}\label{eq13}  T_{ \sA_{g, n}} \stackrel{\sim}{\to} \; S^2 {\rm Lie} {\sX } , \end{equation} 
 and its pull-back to $\mathfrak H_g$ is the restriction of the canonical isomorphism  
  \begin{equation}\label{eq14}  T_{\mathfrak H_{g }^\vee} \stackrel{\sim}{\to} \; S^2 \sL    \end{equation} \cf \eg \cite{G}\cite{CG}.
 
   Any principally polarized abelian scheme with level $n$ structure $A/S$ is isomorphic to the pull-back of $\sX$ by a morphism $S\stackrel{\mu}{\to} \sA_{g, n}$, and the Kodaira-Spencer map of $A/S$ (in the form of \eqref{eq12}) is the pull-back by $\mu$ of the isomorphism \eqref{eq13} composed with $d\mu : T_S \to \mu^\ast T_{ \sA_{g, n}}$. In particular, {\it the Kodaira-Spencer map $\theta$ of $A/S$ vanishes if and only if the image of $S\to  \sA_{g, n}$ is a point, \ie $A/S$ is constant; moreover, if $A/S$ is not constant, $\mu$ is generically finite, and $\partial$ is a non zero section of $T_S$, then $\theta_\partial$ is non zero}.

     \subsection{Relative period torsor} \subsubsection{} The {\it bi-algebraicity} mentioned above refers to the pair of algebraic structures $ \sA_{g, n}, \frak H_g^\vee$, which are transcendentally related via $\frak H_g$ and $j_{g,n}$.

     On the other hand, there is a purely algebraic relation between these two algebraic structures, through the {\it relative period torsor}. This is the $Sp(\Lambda_\C)_{\sA_{g, n}}$-torsor $\Pi_{g,n}\stackrel{\pi}{\to} \sA_{g, n}$ of solutions of the Gauss-Manin connection $\nabla$ of $\sX$. More formally, this is the torsor of isomorphisms 
     $ \sH \to \Lambda_\C^\vee \otimes \sO_{\sA_{g, n}}$ which respect the $\nabla$-horizontal tensors\footnote{this is a reduction from $GSp$ to $Sp$ of the standard principal bundle considered in \cite[III.3]{M}. }. 
     Its generic fiber is the spectrum of the Picard-Vessiot algebra\footnote{in general,  one has to adjoin the inverse of the wronskian together with the entries of a full solution matrix in order to build a Picard-Vessiot algebra, but here the wronskian is a rational function, since the monodromy is contained in $Sp$.}  attached to $\nabla$, namely   
    $\Spec\, \C(\sA_{g,n})[Y_{ij}]_{i,j= 1,\ldots, 2g} \, $ (with the notation of  \ref{GM}).
                
        \subsubsection{}   The canonical horizontal isomorphism $ \sH\otimes _{\sO_{\sA_{g, n}}} \sO_{\frak H_g} \stackrel{\sim}{\to} \Lambda_\C^\vee\otimes_\C  \sO_{\frak H_g}$ gives rise to an analytic map 
      \begin{equation}\label{eq14i}  k: \,\frak H_g \to \Pi_{g,n} \end{equation}  with $\pi \circ k = j_{g,n}$.  In local bases and with the notation of \ref{GM}, $k$ sends $\tau\in \frak H_g$ to the point $Y(\tau) =\begin{pmatrix}  \Omega_1(\tau) & {\rm N}_1(\tau)\\ \Omega_2(\tau) & {\rm N}_2(\tau) \end{pmatrix}$ of $\Pi_{j_{g,n} (\tau)}$. In particular, the image of $k$ is Zariski-dense in $ \Pi_{g,n}$.
    
    On the other hand there is an algebraic $Sp(\Lambda_\C)$-equivariant  map 
        \begin{equation}\label{eq14ii}  \rho : \,\Pi_{g,n}  {\to} \frak H_g^\vee \end{equation}
          which sends a point $p\in  \Pi_{g,n}(\C)$ viewed as an isomorphism  $ \sH_{\pi(p)} \to \Lambda^\vee $ to the image of $ \varOmega_{\sX_{\pi(p)}}$ in $ \Lambda_\C^\vee $. In local basises and with the notation of \ref{GM},  $\rho$ sends $\begin{pmatrix}\Omega_1 &  {\rm N}_1\\ \Omega_2 & {\rm N}_2 \end{pmatrix}$ to $ \tau = \Omega_1 \Omega_2^{-1}$; $ \rho \circ k  $ is the Borel embedding $\frak H_g\inj \frak H_g^\vee$.

         One thus has the following diagram
           \begin{equation}\label{eq14ii}   \frak H_g \to \Pi_{g,n } \stackrel{(\pi,\rho)}{\surj} \sA_{g,n} \times \frak H_g^\vee \end{equation}
         in which the first map has Zariski-dense image, and the second map is surjective (of relative dimension $\frac{g(3g+1)}{2}$) since the restriction of $\rho$ to any fiber of $\pi$ is $Sp(\Lambda_\C)$-equivariant and $\frak H_g^\vee$ is homogeneous.  
         It follows that the graph of $j_{g,n}$ is Zariski-dense\footnote{this property does not extend to $m$-jets for $m\geq 3$.} in $   \frak H_g^\vee \times \sA_{g,n} $.

      The function field of $\Pi_{g,n}$ is studied in detail in \cite{BZ}: it is a differential field both for the derivations of $\sA_{g,n}$ and for the derivations $\partial /\partial \tau_{ij}$ of $\frak H_g^\vee$. Over $\C(\frak H_g^\vee) = \C(\tau_{ij})_{i\leq j\leq g} $, it is generated by (iterated) derivatives with respect to the $\partial /\partial \tau_{ij}$'s of the modular functions (the field of modular functions being $\C(\sA_{g,n})$).

     \subsection{Connected Shimura varieties and weakly special subvarieties} 
     
     \subsubsection{} Let $G$ be a reductive group over $\Q$, $G^{ad}$ the quotient by the center, and $G^{ad}(\R)^+$ the connected component of identity of the Lie group $G^{ad}(\R)$.

     Let $X$ be a connected component of a conjugacy class $\frak X$ of real-algebraic homomorphisms $\C^\ast \to G_\R$. For any rational representation $G \to GL(W)$, one then has a collection of real Hodge structures $(W_\R, h_x)_{x\in X}$ on $W_\R$ parametrized by $X$. If the weight is defined over $\Q$ (which is the case if $G= G^{ad}$ since the weight is $0$ in this case), one even has a collection of rational Hodge structures $(W , h_x)_{x\in X}$ . 
     
   In the sequel, we assume that $(G, \frak X)$ satisfies Deligne's axioms for a {\it Shimura datum}; these axioms ensure that $X$ has a $G^{ad}(\R)^+$-invariant metric, which makes $X$ into a hermitian symmetric domain, and that the $(W_\R, h_x)$ (\resp $(W , h_x)$) come from variations of polarized Hodge structures on the analytic variety $X$ (\resp if the weight is defined over $\Q$, for instance if $G=G^{ad}$); moreover, in the case of the adjoint representation on $\frak g = {\rm Lie}\, G $, the variation of Hodge structures is of type $(-1,1)+(0,0)+(1,-1)$ (\cf \eg \cite[II]{M}).

       \subsubsection{} Let $\Gamma$ be a discrete subgroup of $G^{ad}(\Q)^+$, quotient of a torsion-free congruence subgroup of $G(\Q)$. Then $ \Gamma \backslash X$ has a canonical structure of algebraic variety (Baily-Borel):   the {\it connected Shimura variety} attached to $(G, X, \Gamma)$. The variation of Hodge structures descends to it, with monodromy group $\Gamma.$ The situation of 2.1 corresponds to the case $G= GSp_{2g}, \; X= \frak H_g, \; \Gamma = $ the congruence subgroup of level $3$ (\cf \eg \cite[II]{M}).
     
      \subsubsection{} Let $S$ be the connected Shimura variety associated to attached to $(G, X, \Gamma)$, and $j: X \to S$ the uniformizing map.   An irreducible subvariety $S_1\subset S$ is {\it weakly special} if there is a sub-Shimura datum $(H, \frak Y)\to (G, \frak X)$, a 
decomposition $(H^{ad}, \frak Y^{ad}) = (H_1, \frak Y_1)\times (H_2, \frak Y_2)$, and a point $y\in \frak Y_2$ such that $S_1$ is the image of  $Y_1 \times y$ in $S$ (here $Y_1$ is a connected component of $\frak Y_1$ contained in $X$)\cite{UY}\footnote{this is a {\it special} subvariety if $y$ is a special point. }; in particular, $S_1$ is isomorphic to the connected Shimura variety attached to $(H_1, Y_1, \Gamma^{ad}\cap H_1)$.

  \newpage   \subsection{Automorphic vector bundles} 
     
     \subsubsection{} Given a faithful rational representation $W$ of $G$, the associated family of Hodge filtrations on $W_\C$ is parametrized by a certain flag variety $X^\vee$, the {\it compact dual} of $X$, which is a $G^{ad}_\C$-homogeneous space.
     
     The isotropy group of a point $x \in X^\vee$ is a parabolic subgroup $P_x$,  $K_x := P_x \cap G^{ad}(\R)^+$ is a maximal compact subgroup, and there is a $G^{ad}(\R)^+$-equivariant {\it Borel embedding} 
      \begin{equation}\label{eq16} X = G^{ad}(\R)^+/ K_x \stackrel{i}{\inj} X^\vee = G^{ad}_\C/ P_x. \end{equation}  
         
  \subsubsection{} Associated to $W$, there is a variation of polarized Hodge structures on $S = \Gamma \backslash X $, hence an integrable connection $\nabla$ with regular singularities at infinity on the underlying vector bundle  $ {\sW}$ . There is again a relative period torsor in this situation. 
  
  Assume for simplicity that $G=G^{ad}$. The monodromy group $\Gamma$ is then Zariski-dense in $G$. 
     The {\it relative period torsor} 
        \begin{equation}\label{eq16i}  \Pi\stackrel{\pi}{\to} S \end{equation} is the $G_S$-torsor of isomorphisms 
     $ \sW \to W_C \otimes \sO_{S}$ which respects the $\nabla$-horizontal tensors\footnote{it coincides with the standard principal bundle considered in \cite[III.3]{M}.}.  Its generic fiber is the Picard-Vessiot algebra attached to $\nabla$. 
        
    The canonical horizontal isomorphism $ \sW\otimes _{\sO_{S}} \sO_{X} \stackrel{\sim}{\to} W_\C \otimes  \sO_{X}$ gives rise to an analytic map 
       $k: X \to \Pi $ with $\pi \circ k = j $.   
 There is an algebraic $G_\C$-equivariant  map $\Pi  \stackrel{\rho}{\to} X^\vee$ (which sends a point $p\in \Pi(\C)$ viewed as an isomorphism  $ \sW_{\pi(p)} \to W$ to the point of $X^\vee$ which parametrizes the image of the Hodge filtration of $ \sW_{\pi(p)}  $);  one has $ \rho \circ k = i $.
 
         One thus has the following factorization:
           \begin{equation}\label{eq16ii} (j, i):\; X \to \Pi \stackrel{(\pi,\rho)}{\surj} S \times X^\vee \end{equation}
         in which the first map has dense image, the second map is surjective (since the restriction of $\rho$ to any fiber of $\pi$ is $G_\C$-equivariant with homogeneous target).  
        
\smallskip
 
 Since any faithful rational representation of $G$ lies in the tannakian category generated by $W$ and conversely, neither $X^\vee$ nor $\Pi$ depend on the auxiliary $W$. On the other hand, $\pi^\ast$ provides an equivalence between the category of vector bundles on $S$ and the category of $G_\C$-vector bundles on $\Pi$ \cite[III.3.1]{M}.  
  
    \subsubsection{}  A $G_\C$-equivariant vector bundle $\breve  \sV$ on $X^\vee = G_\C/ P_x$ is completely determined by its fiber at $x \in X$ together with the induced $P_x$-action (or else, the induced $K_x$-action).  The quotient $\sV := \Gamma \backslash i^{\ast}\breve  \sV$ has a canonical structure of algebraic vector bundle on $S = \Gamma \backslash X$, the {\it automorphic vector bundle} attached to $\breve \sV$ \cite[III.2.1, 3.6]{M}. One has the equality of analytic vector bundles on $X$:
    \begin{equation}\label{eq17} j^{\ast}\sV = i^{\ast}\breve \sV . \end{equation}  
      There is also a purely algebraic relation between $\sV$ and $\breve  \sV$, through the relative period torsor \cite[III.3.5]{M}: one has the equality of algebraic $G_\C$-vector bundles on $\Pi$:
    \begin{equation}\label{eq18} \pi^\ast \sV = \rho^\ast \breve  \sV. \end{equation}   
    
    \subsubsection{} Any representation of $G_\C$ gives rise to a $G_\C$-equivariant vector bundle on $X^\vee$, hence to an automorphic vector bundle (which carries an integrable connection).
    
    On the other hand, $T_{X^\vee}$ is a $G_\C$-equivariant vector bundle on $X^\vee$, and the corresponding automorphic vector bundle is nothing but $T_S$.

   In the situation of \ref{ks},  The  tangent bundle $T_{\frak H_g^\vee}$ and its tautological bundle $\sL$ are equivariant vector bundles, and the universal Kodaira-Spencer map \eqref{eq13} is an isomorphism of automorphic vector bundles on $ \sA_{g, n}$.

      \subsection{A theorem of logarithmic Ax-Schanuel type for tangent bundles}
  
   \subsubsection{}\label{swss}   The theorem of logarithmic Ax-Schanuel for connected Shimura varieties is the following \cite[2.3.1]{Ga} (\cf also \cite{UY})\footnote{not to be confused with the (exponential) Ax-Schanuel theorem - a.k.a. hyperbolic Ax-Lindemann - for connected Shimura varieties, which concerns the maximal irreducible algebraic subvarieties of $X^\vee$ whose intersection with $X$ is contained in $\tilde Z$, and which is a much deeper result (Pila-Tsimerman, Klingler-Ullmo-Yafaev).}:
  
    \begin{thm}\label{t3} Let $S$ be a connected Shimura variety ($S^{an} = \Gamma \backslash X$). Let $Z\subset S$ be an irreducible locally closed subset, and let $\tilde Z$ be an analytic component of the inverse image of $Z$ in $X$. 
  
  Then the image in $S$ of the Zariski closure of $\tilde Z$ in the compact dual $X^\vee$ is the smallest weakly special subvariety  $S_1\subset S$ containing $Z$.    \end{thm}

   Here is a sketch of proof. One can replace $S$ by the smallest special subvariety containing $Z$. Fix a point $s\in Z(\C)$ and a faithful rational representation of $G$, and consider the associated vector bundle $\sW$ with integrable connection $\nabla$ on $S$. Let $\hat G_1\subset G$ be the Zariski closure of the monodromy group $\Gamma_Z$ of $(\sW_{\mid Z}, \nabla_{\mid Z})$ at $s$. Up to replacing $\Gamma$ by a subgroup of finite index, $\hat G_1$ is connected and a normal subgroup of $G$  (by \cite[5]{A}). This gives rise to a weakly special subvariety $S_1\subset S$ associated to a factor $G_1= \hat G_1^{ad}$ of $G^{ad}$, which is in fact {\it the smallest weakly special subvariety of $S$ containing $Z$} 
    (\cf \cite[3.6]{Mo}, \cite[4.1]{UY} for details). On the other hand, since $\tilde Z$ is stable under $\Gamma_Z$, its Zariski closure in the $G_\C$-homogeneous space $X_1^\vee$ is stable under $G_1$, hence equal to $X_1^\vee$.   \qed

   Here is the analog for tangent vector bundles, assuming $Z$ smooth:
   
       \begin{thm}\label{t4} In this situation,  $T_{\tilde Z}$ is Zariski-dense in $T_{X_1^\vee}$. 
    \end{thm}
   
     \begin{proof}   We may replace $G$ by $G_1$ and $S$ by $S_1$.  Let $\overline{T_{\tilde Z}}$ be the Zariski closure of $T_{\tilde Z} =   \tilde Z \times_Z T_Z$ in $T_{X^\vee}$. Let $(\sW, \nabla)$ be as above, and let $\Pi $ be the relative period torsor of $S $ (we take over the notation \eqref{eq14i}  \eqref{eq14ii}).  Since $(\sW_{\mid Z}, \nabla_{\mid Z})$ has the same algebraic monodromy group as $(\sW, \nabla)$, namely $G$,  the generic fiber of the projection $\Pi_Z\stackrel{\pi_Z}{\to} Z$ is the spectrum of the Picard-Vessiot algebra attached to $(\sW_{\mid Z}, \nabla_{\mid Z})$, and the image of $k_{\mid \tilde Z}(\tilde Z)$ in Zariski-dense in $\Pi_Z$. It follows that   $(k_{\mid \tilde Z}\times 1_{T_Z})( T_{\tilde Z})= k_{\mid \tilde Z}(\tilde Z)\times_Z T_Z$ is Zariski-dense in $\Pi_Z\times_Z T_Z $. 
      
       Let us consider the composition of algebraic maps 
 $$ \tau:\; \Pi_Z\times_Z T_Z   =\Pi \times_S T_Z \to \Pi \times_S T_S = \Pi \times _{X^\vee} {T_{\tilde Z}} \to {T_{\tilde Z}}$$ (the second equality uses the fact that $T_S$ is an automorphic vector bundle with $\breve T_S = T_{X^\vee}$, cf \eqref{eq18}). 
  The inclusion $T_{\tilde Z}\to  T_{X^\vee} $ factors as $\tau \circ (k_{\mid \tilde Z}\times 1_{T_Z})$. Hence $\overline{T_{\tilde Z}}$ is the Zariski closure of the image of $\tau$.   
     
      On the other hand,  $\Pi_Z\stackrel{\pi_Z}{\to} Z$ induces a surjective map $T_{\Pi_Z}\to \Pi_Z\times_Z T_Z$, whose composition with $\tau$ is the tangent map 
      $   T_{\Pi_Z} \stackrel{T_{\upsilon}}{\to} T_{X^\vee}$  of $$\upsilon: \Pi_Z\to X^\vee.$$ Hence $\overline{T_{\tilde Z}}$ is also the Zariski closure of the image of $T_\upsilon $. Since $\upsilon $ is equivariant with homogeneous target, it is surjective, hence $T_\upsilon$ is dominant. One concludes that  $\overline{T_{\tilde Z}} =  T_{X^\vee}$.
          \end{proof} 
          
          This also holds, with the same proof, for higher jets bundles $J^m\tilde Z\subset J^m X_1^\vee$.
         
         \begin{rem} In general, given an algebraic vector bundle $\sM$ on an algebraic variety $Y$, the Zariski closure of an analytic subbundle  over some Zariski-dense analytic subspace of $Y$ is not necessarily an algebraic subbundle of $\sM$: for instance, the Zariski closure in $T_{\C^2}$ of the tangent bundle of the graph in $\C^2$ of a Weierstrass $\wp$ function is a quadric bundle over ${\C^2}$, not a vector subbundle of $T_{\C^2}$ (a similar counterexample holds for the graph of the usual $j$-function and its bundle of jets of order $\leq 3$, since $j$ satisfies a rational non-linear differential equation of order $3$). 
         
         On the other hand, theorem \ref{t4} does not extend to arbitrary automorphic vector bundles: it may happen that the Zariski closure of the pull-back on $\tilde Z$ of a vector subbundle of $T_Z$ is not vector subbundle of $T_{S_1}$. However, one has the following easy consequence of \ref{t3}:          \end{rem}
     
     \begin{por} In the same situation, let $\sV$ be an automorphic vector bundle on $S$ with corresponding vector bundle $\breve \sV$ on $X^\vee$, and let $\sF$ be a vector subbundle of the restriction of $\sV $ to $Z$. Then $Z$ and $\sF$ are bi-algebraic if and only if $Z$ is a weakly special subvariety and $\sF$ is an automorphic vector bundle.
 \end{por} 
 
The assumption ``$Z$ is bi-algebraic" means that $\tilde Z$ is the intersection of $X$ with an algebraic subvariety of $X^\vee$, and according to \ref{t3}, this amounts to $Z=S_1$.
 
The assumption ``$\sF$ is bi-algebraic" means that its pull-back $\tilde \sF$ in  $\breve \sV$ is an algebraic subvariety. Since $Z=S_1$, this amounts to say that this analytic subbundle of  $  \sV\times_S {X_1^\vee} = \breve \sV\times_{X^\vee} X_1^\vee$ is algebraic. It is in fact a $G_{1\C}$-vector subbundle, so that $\sF$ is an automorphic vector bundle on  $ X_1^\vee$.     \qed
      
     Using the relative period torsor, one can also prove the following stronger version of \ref{t3}: 
     
  \begin{sco}  in the setting of \ref{t3}, the graph of $j_{\mid \tilde Z}$ is Zariski-dense in $X_1^\vee \times Z$. 
  \end{sco} 
 
Indeed, it follows from \eqref{eq16ii} (with $S_1$ in place of $S$) that the map $\Pi_{1\mid Z} \stackrel{(\rho, \pi_Z)}{\to} X_1^\vee \times Z$ is surjective. On the other hand, the image of $k_{\mid \tilde Z}(\tilde Z)$ is Zariski-dense in $\Pi_{1\mid Z}$.  \qed

    \section{Transition to the modular case}

   We go back to the study of $r(A/S), \, r'(A/S), \, r''(A/S)$ for an abelian scheme $A/S$.  We may assume that the base field $k$ is $\C$. According to lemma \ref{l2}, these ranks are invariant by dominant base change of $S$ and by isogeny of $A$, hence one may assume that $A/S$ admits a principal polarization and a Jacobi level $n$ structure for some $n\geq 3$, and then replace $S$ by the smooth locus $Z\subset \sA_{g, n}$ of its image in the moduli space of principally polarized abelian varieties of dimension $g$ with level $n$ structure, and $A$ by the restriction $\sX_Z$ of the universal abelian scheme $\sX$ on  $ \sA_{g, n}$.

 \subsection{From $Z$ to the smallest weakly special subvariety of $ \sA_{g, n}$ containing $Z$} 
      
    Let us consider again the situation of \ref{swss}, with $S = \sA_{g, n}$. Given a (locally closed) subvariety $Z\subset \sA_{g, n}$, one constructs the smallest weakly special subvariety $ S_1\subset  \sA_{g, n}$ containing $Z$, taking  $(\sW, \nabla)$ equal to $\sH^1_{dR}(\sX/\sA_{g, n})$ with its Gauss-Manin connection. By construction, $ S_1(\C) = \Gamma_1 \backslash X_1$ where $X_1$ is a hermitian symmetric domain attached to the adjoint group $G_1$ of the connected algebraic monodromy group of $\nabla_{\mid Z}$. 
       
      \begin{thm}\label{t5} One has $r(\sX_Z/ Z) = r(\sX_{S_1}/ S_1)$ and $r''(\sX_Z/ Z) = r''(\sX_{S_1}/ S_1)$. 
      \end{thm}
      
  \begin{proof} Fix $s\in Z(\C)$. By construction $\nabla_{\mid Z}$ and $\nabla_{\mid S_1}$ have the same connected algebraic monodromy group at $s$, namely $\hat G_1\subset Sp_{2g}$ (up to replacing $n$ by a multiple). It follows that $\sD_Z \varOmega_{\sX_Z} =   (\sD_{S_1} \varOmega_{\sX_{S_1}})_{\mid Z}$, whence $r(\sX_Z/ Z) = r(\sX_{S_1}/ S_1)$.
  
  \smallskip On the other hand, the inequality $r''(\sX_Z/ Z) \leq r''(\sX_{S_1}/ S_1)$ is obvious. For any natural integer $h < g$, let $\Delta_h$ be the closed subset of $T_{\mathfrak H_g^\vee}$ corresponding to quadratic forms in $S^2 {\rm Lie} {\sX }$ of rank $\leq  h\,$ (this is in fact a $Sp(\Lambda_\C)$-subvariety; $\Delta_{0}$ is the $0$-section). Then $r''(\sX_Z/ Z) $ (\resp $r''(\sX_{S_1}/ S_1) $) is the greatest integer $h$ such that $d\mu(\partial)$ in not contained in $\mu^\ast\Delta_{h-1}$.  In order to prove the inequality $r''(\sX_Z/ Z) \geq r''(\sX_{S_1}/ S_1)$, it thus suffices to show that if $T_{\tilde Z}$ is not contained in $\Delta_h$, neither is $T_{X_1^\vee}$, which follows from the fact that $T_{\tilde Z}$ is Zariski-dense in $T_{X_1^\vee}$ (\ref{t4}).
    \end{proof}

     \subsection{Case of a weakly special subvariety of $\sA_{g,n}$}

We now assume that $S$ is a weakly special subvariety of $\sA_{g,n}$, with associated group $G=G^{ad}$, and that there is a finite covering $\hat G $ of $G$ contained in $Sp(\Lambda_\Q)$. 

  \begin{thm}\label{t6} One has ${\rm Im}\, \theta  = \sD_S \varOmega_{\sX_S}/  \varOmega_{\sX_S}$, hence $r(\sX_{S}/ S) = r'(\sX_{S}/ S)$. 
      \end{thm}
  
\begin{proof} 
 Fix an arbitrary point $x\in X$ and set $s= j (x)\in S$. Then $X  = G(\R) /K_x  $, and $X^\vee = G_\C /P_x  $ can also be written $\hat G_\C / \hat P_x \subset \frak H_g^\vee$;   $\hat P_x$ stabilizes the lagrangian subspace $V_x := \varOmega_{\sX_s} \subset \Lambda^\vee_\C$. We write
  $$\frak g= {\rm Lie}\, G_\C = {\rm Lie}\, \hat G_\C ,\;\; \frak k_\C   =   {\rm Lie}\, K_{x,\C}.$$  
   The Hodge decomposition of $\frak g$ with respect to $ad\circ h_x$ takes the form $\frak u^+ \oplus \frak k_\C \oplus \frak u^-$, where $\frak u^+$, the Lie algebra of the unipotent radical of $P_x$, is of type $(-1,1), \; \frak k_\C$ is of type $(0,0)$, and $\frak u^-$ of type $(1,-1)$. One has $\frak u^+ \oplus \frak k_\C = {\rm Lie}\, \hat P_x$, $  \frak k_\C \oplus \frak u^-$ is the Lie algebra of an opposite parabolic group $P_x^-$, and $ \frak k_\C$ is the common (reductive) Levi factor (\cf also \cite[5]{Mo}). 
 
 Looking at the Hodge type, one finds that 
 \begin{equation}\label{eq18}[\frak u^+ ,\frak u^+ ]= [\frak u^- ,\frak u^- ]= 0, \; [\frak k_\C ,\frak u^+ ]\subset  \frak u^+ ,\;   [\frak k_\C ,\frak u^- ]\subset  \frak u^-, \;.    [\frak u^+, \frak u^-] \subset \frak k_\C. \end{equation}
By the Jacobi identity, it follows that
  \begin{equation}\label{eq19}[\frak k_\C, [  \frak u^+, \frak u^-]] \subset  [  \frak u^+, \frak u^-],  \end{equation}
\ie $[\frak u^+, \frak u^-] $ is a Lie ideal of $\frak k_\C$, hence a reductive Lie algebra.

  We may identify $T_xX^\vee = T_s S$ with $\frak u^-$. Note that $\Lambda^\vee_\C$ is a faithful representation of $\frak g$ and that $V_x$ is stable under ${\rm Lie}\, P_x = \frak u^+ +\frak k_\C $. Using the Hodge decomposition $V_x \oplus \bar V_x = \Lambda^\vee_\C \cong H_{dR}^1(\sX_s)$, we can write the elements of $\frak g$ as matrices in block form $\begin{pmatrix} R&S\\T&\iota(R) \end{pmatrix} $, with $R \in \frak k_\C, \;S\in \frak u^+, \, T\in  \frak u^-$ and $\iota$ is the involution exchanging $P_x$ and $P^-_x$. 
Identifying $\frak u^+$ with $  \begin{pmatrix} 0  &  \frak u^+ \\    0 &  0 \end{pmatrix}$ (\resp $\frak u^-$ with $  \begin{pmatrix} 0  &  0 \\    \frak u^- &  0 \end{pmatrix}$), one may write $ [\frak u^+, \frak u^-] = \begin{pmatrix}  \frak u^+\cdot\frak u^-  & 0\\  0 &\iota(\frak u^+\cdot\frak u^-) \end{pmatrix}$. Therefore $\frak u^+\cdot\frak u^-$ is a reductive Lie algebra acting on $V_x$. Accordingly,  $V_x$ decomposes as $V_0 \oplus V'$, where $V_0$ is the kernel of this action, and $(\frak u^-\cdot\frak u) V' = V'$.  

The identifications $\frak u^-= T_s S$ and $V_x = (\varOmega_{\sX_S})_s$ lead to $V_x \oplus \frak u^- V_x =  (\sD^{\leq 1}_S \varOmega_{\sX_S})_s$. 

\smallskip {\it Claim: $V_x \oplus \frak u^- V_x$ is the smallest $\frak g$-submodule of $\Lambda^\vee_\C$ containing $V_x$. Therefore, it is the fiber at $x$ of $\sD_S \varOmega_{\sX_S}$.}

The point is that $V_x + \frak u^- V_x = V_x + \frak u^- V' $ is stable under $\frak u^-, \frak u^-$ and $\frak k_\C$, which follows from \eqref{eq18} and from the fact that $V_x$ is stable under $\frak u^+ +\frak k_\C $.  
 \end{proof}

   \section{The case of maximal monodromy (subject to given polarization and endomorphisms)} 
   
   \subsection{Abelian schemes of PEM type} 
    
   \begin{defn}\label{pe} A principally polarized abelian scheme $ A/S$ is {\it of PE-monodromy type - or PEM type -} if its geometric generic fibre is simple and the connected algebraic monodromy is maximal with respect to the polarization $\psi$ and the endomorphisms. 
     \end{defn} 
     
      In other words, the Zariski-closure of the monodromy group at $s\in S$ is the maximal algebraic subgroup of $Sp(H^1(A_s), \psi_s)$ which commutes with the action of ${\rm{End}}\,A/S$ (this condition is independent of $s\in S(\mathbb C)$).
     
   \medskip  Let us make this more explicit. The endomorphism $\Q$-algebra $D := ({\rm{End}}\,  A/S)\otimes \Q$ is the same as the one of its generic fiber; since the latter is assumed to be geometrically simple, $D$ is also the endomorphism $\Q$-algebra of the geometric generic fibre. According to Albert's classification, its falls into one of the following types:
     
     I :  a totally real field $F= D$,
     
     II:  a totally indefinite quaternion algebra $D$ over a totally real field $F$,
     
     III: a totally indefinite quaternion algebra $D$ over a totally real field $F$,

     IV: a division algebra $D$ over a CM field $F$. 
     
  \noindent  Let $E\supset F$ be a maximal subfield of $D$, which we can take to be a CM field except for type I, and let $E^+$ be a maximal totally real subfield.
  For any embedding $\lambda: \,E^+ \inj \R$, let us order the embeddings $ \lambda_1, \lambda_2 : E \inj \C $ above $\lambda$ if $E\neq E^+$ (and set $\lambda_1 = \lambda$ otherwise). We identify $\lambda$ (\resp $\lambda_1$) with a homomorphism $E^+\otimes \C \to \C$ (\resp $E \otimes \C \to \C$). 
  Let us set 
    \begin{equation}\label{eq20} \sH_{\lambda} = \sH \otimes_{E^+\otimes \C, \lambda} \C \;\;  (\resp  \; \sH_{\lambda_1} = \sH \otimes_{E\otimes \C , \lambda_1} \C).  \end{equation} 
    By functoriality of the Gauss-Manin connection, these are direct factors of $\sH$ as $\sD_S$-modules, and $\sH_{\lambda} $ only depends (up to isomorphism) on the restriction $[\lambda]$ of $\lambda$ to $F^+$. 
       
  Then the maximal possible connected complex monodromy group at an arbitrary point $s\in S(\C)$ is of the form 
  $\Pi_{[\lambda]} \, G_{[\lambda]}$ where $G_{[\lambda]}$ and its representation on $\sH_{\lambda, s}$ are of the form 
  
  I:  $Sp(\sH_{\lambda, s}),\; St $,
  
  II:  $Sp(\sH_{\lambda_1, s}),\; St \oplus St$,
  
  III:  $SO(\sH_{\lambda_1, s}),\; St \oplus St$,
  
  IV: $SL(\sH_{\lambda_1, s}), \; St \oplus St^\ast$,
 
  \noindent  where $St$ denotes the standard representation, and $St^\ast$ its dual. Moreover, for types I, II, III, $\sH_{\lambda_1, s}$ is an even-dimensional space, \cf \eg \cite[5]{Ab}. 
  
   \begin{rem} If $A/S$ is endowed with a level $n$ structure, it is of PEM type if and only if the smallest {\it weakly special subvariety} of $\sA_{g,n}$ containing the image of $S$ is a special subvariety of PEL type in the sense of Shimura, \ie the image in $\sA_{g,n}$ of the moduli space for principally polarized abelian varieties $A$ such that $D\subset ({\rm{End}}\,A)\otimes \Q$, equipped with level $n$ structure \cite{S} (for $S = \Spec k$, $A$ is of PEM type if and only if $A$ has complex multiplication). 
   
  One could also define the related (but weaker) notion of abelian scheme $A/S$ of PE Hodge type, on replacing the monodromy group by the Mumford-Tate group, \cf \eg \cite{Ab}. If $A/S$ is endowed with a level $n$ structure, it is of PE Hodge type if and only if the smallest {\it special subvariety} of $\sA_{g,n}$ containing the image of $S$ is a special subvariety of PEL type in the sense of Shimura.
\end{rem}
     
     \subsubsection{}  Parallel to \eqref{eq20}, one has a decomposition
       \begin{equation}\label{21} \varOmega_{A, \lambda} = \varOmega_A \otimes_{E^+\otimes \C, \lambda} \C \;\;  (\resp  \; \varOmega_{A, \lambda_1} = \varOmega_A \otimes_{E\otimes \C , \lambda_1} \C).  \end{equation} 
       The sequence \eqref{eq1} 
    induces an exact sequence  
   \begin{equation}\label{eq22} 0 \to \varOmega_{A, \lambda_1}  \to  \sH_{\lambda_1}  \to   \sH_{\lambda_1} / \varOmega_{A, \lambda_1} \to 0.\end{equation} It turns out that for 
 types I, II, III,  $ \sH_{\lambda_1} / \varOmega_{A, \lambda_1} \cong  \varOmega^\vee_{A, \lambda_1}$.
        This is not the case for type $IV$, and the pair 
        \begin{equation}\label{eq23}  (r_{[\lambda]} = \dim \varOmega_{A, \lambda_1, s} , \; s_{[\lambda]} = \dim   \sH_{\lambda_1, s}/\varOmega_{A, \lambda_1, s} ) \end{equation} is an interesting invariant called the {\it Shimura type} (for type IV, the PEL families depend not only on $D$, the polarization and the level structure, but also on these pairs, when $[\lambda]$ runs among the real embeddings of $F^+$).  
       On the other hand,  $\varOmega_{A, \lambda_2} \cong \varOmega_{A, \lambda_1}$ for types  I, II, III, while $\varOmega_{A, \lambda_2} \cong    \sH_{\lambda_1} / \varOmega_{A, \lambda_1} $ for type IV.

   \subsubsection{}\label{exe}  By functoriality, the Kodaira-Spencer map induces a map 
    \begin{equation}\label{eq24}  \theta_{\partial, \lambda_1}: \; \varOmega_{A, \lambda_1}  \to     \sH_{\lambda_1} / \varOmega_{A, \lambda_1}.\end{equation}
    Therefore,      
     $ {\rm rk} \, \theta_{\partial, \lambda_1} \leq   \min (r_{[\lambda]} , s_{[\lambda]})$. In particular, if for some $[\lambda]$, $r_{[\lambda]} \neq s_{[\lambda]}$, then $r'< g$. 
     
     Let us consider for example the Shimura family of PEL type of abelian $3$-folds with multiplication by an imaginary quadratic field $E$ (type IV) and invariant $(r= 1, s=2)$ (it is non empty by \cite{S}). The base is a Shimura surface. Let $A/S$ be the restriction of this abelian scheme to a general curve of this surface. Then $r''= 1, \, r'= 2,\, r= g= 3$. 
     
     One gets examples with $r <  g$ when $r_{[\lambda]}\cdot s_{[\lambda]} = 0$ for some ${[\lambda]}$.

     \subsection{Abelian schemes of restricted PEM type}   
     
        \begin{defn}\label{rpe} A principally polarized abelian scheme $ A/S$ is of {\it restricted PEM type} if it is of PEM type and for any (equivalently, for all) $s\in S(\C)$, $(\varOmega_{A})_s$ is a free $E\otimes \C$-module.
     \end{defn} 
     
    In the latter condition, one could replace $E$ by $F$. It is automatic for types I, II, III. For type IV, it amounts to the equality $r_\lambda = s_\lambda$ for every $\lambda$; in that case, $\Omega_{\lambda_1} \cong \Omega_{\lambda_2}^\vee$.
   
     \begin{thm}\label{t7} In the restricted PEM case, one has $r''=r'=r=g$. 
      \end{thm}
 
 \begin{proof} Thanks to lemma \ref{l2}  and theorem \ref{t5}, we are reduced to prove that $r'' = g$ for a Shimura family of PEL type, provided $r_\lambda= s_\lambda$ for every $\lambda$ in the type IV case.  This amounts in turn to showing that there exists $\partial$ such that $\theta_{\partial, \lambda_1}$ has maximal rank, equal to the rank of $\sH_{\lambda_1}$ which is twice the rank $m$ of $\varOmega_{A, \lambda_1}$.
 Let $\frak g$ be one of the Lie algebras $sp(2m), \, so(2m), \, sl(2m)$. In the notation of the proof of theorem \ref{t6}, The point is to show that $\frak u^-$ contains an invertible element. But $\frak u^+$  consists of lower left quadrants of elements of $\frak g$ viewed as a $2m$-$2m$-matrices; and it is clear that the lower left quadrant of a general element of $\frak g$ is an invertible $m$-$m$-matrix.
 \end{proof}
 
 \begin{rems} $i)$ One can be more precise and give an interpretation of the partial Kodaira-Spencer map at the level of $X^\vee$  as induced by isomorphisms 
    \begin{equation}\label{eq25} T_{G_{[\lambda]/P_{[\lambda]}}} \stackrel{\sim}{\to} S^2 \sL_{\lambda_1} \end{equation} for type I and II (the lower left quadrant of an element of $\frak g$ is symmetric), 
        \begin{equation}\label{eq26} T_{G_{[\lambda]/P_{[\lambda]}}} \stackrel{\sim}{\to}   \sL_{\lambda_1}^{\otimes 2} \end{equation} for type III and IV (the lower left quadrant of an element of $\frak g$ can be any $m$-$m$-matrix).

 \smallskip $ii)$ Of course one has $r= g$ whenever $\sH$ is an irreducible $\sD_S$-module. 
    
\smallskip  {\it Claim: If ${\rm{End}}_S A = \Z $ and $A/S$ is not isotrivial, then $\sH$ is an irreducible $\sD_S$-module.}

  Indeed, the conclusion can be reformulated as: the local system $R_1f^{an}_\ast \C$ is irreducible. Since we know that it is semisimple \cite[\S 4.2.6]{D}, this is also equivalent, by Schur's lemma, to ${\rm{End}}\, R_1f^{an}_\ast \C = \C $ and also to ${\rm{End}}\, R_1f^{an}_\ast \Z = \Z $. This equality then follows from the assumptions by the results of \cite[\S 4.4]{D}. More precisely, let $Z$ be as in \loccit the center of ${\rm{End}}\, R_1f^{an}_\ast \Q$; then $Z$ is contained in $({\rm{End}}_S A)\otimes \Q$ (\loccit, 4.4.7), hence equal to $\Q$, and by \loccit prop. 4.4.11 (under conditions $(a), (b), (c_1)$ or $(c_2)$), one deduces that ${\rm{End}}\, R_1f^{an}_\ast \Z = \Z $.
       
 \smallskip
 It would be interesting to determine whether $r''=g$ in this case, beyond the PEM case. \end{rems}

\subsection{Differentiating abelian integrals of the first kind with respect to a parameter}  From the above results about differentiating differential forms of the first kind with respect to parameters, it is possible to draw results about differentiating their integrals. 

        An {\it abelian integral of the first kind} on $A$ is a $\C$-linear\footnote{or $\sO(S)$-linear linear, this amounts to the same.} combination of abelian periods $  \int_\gamma \omega $, with $\omega \in \Gamma \varOmega_A$ and $\gamma$ in the period lattice on a universal covering $\tilde S$ of $S^{an}$.  
             
  \begin{thm}\label{t8} Assume that $A$ is an abelian scheme of restricted PEM type over an affine curve $S$.  Let $\partial$ be a non-zero derivation of $\sO(S)$. Then the derivative of a non zero abelian integral of the first kind is never an abelian integral of the first kind (on $A$).   \end{thm}  
  
  \begin{proof} Let us first treat the case when the monodromy of $A/S$ is Zariski-dense in $Sp_{2g}$ for clarity. We may assume that $\varOmega_A$ is free. Then an abelian integral of the first kind is an $\sO(S)$-linear combination  $\sum_{ij} \lambda_{ij} \partial \int_{\gamma_i} \omega_j$ of entries of $\begin{pmatrix} \Omega_1 \\  \Omega_2 \end{pmatrix}$. By \eqref{eq8}, $\sum_{ij} \lambda_{ij} \partial \int_{\gamma_i} \omega_j = \sum_{ijk} \lambda_{ij} (\int_{\gamma_i} \omega_k  (R_\partial)_{kj} +  \int_{\gamma_i} \omega_k  (T_\partial)_{kj})$, \ie
 an $\sO(S)$-linear combination of entries of $\begin{pmatrix} \Omega_1 R_\partial + {\rm N}_1 T_\partial \\  \Omega_2  R_\partial + {\rm N}_2 T_\partial\end{pmatrix}$.  
 
  Since the monodromy of $A/S$ is Zariski-dense in $Sp_{2g}$, $Y = \begin{pmatrix} \Omega_1& {\rm N}_1 \\  \Omega_2 &{\rm N}_2\end{pmatrix}$ is the generic point of a $Sp_{2g, \C(S)}$-torsor, by differential Galois theory in the fuchsian case (Picard-Vessiot-Kolchin-Schlesinger). Since there is no linear relations between the entries of a generic element of $Sp_{2g}$, there is no $\C(S)$-linear relations between the entries of $\begin{pmatrix} \Omega_1& {\rm N}_1 \\  \Omega_2 &{\rm N}_2\end{pmatrix}$, or else between the entries of $\begin{pmatrix} \Omega_1& {\rm N}_1 T_{\partial} \\  \Omega_2 &{\rm N}_2 T_{\partial} \end{pmatrix}$ since $T_\partial$ is invertible (\ref{t7}). One concludes that $\sum_{ij} \lambda_{ij} \partial \int_{\gamma_i} \omega_j =  \sum \mu_{ij}   \int_{\gamma_i} \omega_j $ with $\lambda_{ij}, \mu_{ij}\in \sO(S)$ implies $\lambda_{ij}= \mu_{ij} =0$
  
  \smallskip The other cases are treated similarly, decomposing $\sH$ into pieces of rank $2m$ indexed by $\lambda$ as above, and replacing $Sp_{2g}$ by $Sp_{2m}, \, SO_{2m}$ or $SL_{2m}$ according to the type.
  \end{proof}

 \small{{\it Acknowledgements} - I thank D. Bertrand and B. Moonen for their careful reading.}

     \end{sloppypar}

    \end{document}